\documentclass{article}
\usepackage[english]{babel}
\usepackage{amsmath}
\usepackage{amsfonts}
\usepackage{amsthm}
\usepackage{xcolor}
\usepackage{cancel}

\usepackage{mathabx}

\newcommand{\FF}{{\mathrm{ff}}}
\newcommand{\PM}{{\mathrm{pm}}}
\newcommand{\BJ}{{\mathrm{BJ}}}
\renewcommand{\vec}[1]{{\ensuremath{\boldsymbol{\mathrm #1}}}}
\newcommand{\ten}[1]{\ensuremath{\boldsymbol{\mathsf{#1}}}}

\newcommand{\vdot}{\boldsymbol{\mathsf{\ensuremath\cdot}}}
\newcommand{\del}{\ensuremath{\nabla}}
\newcommand{\deld}{\ensuremath{\del\vdot}}
\newcommand{\lrp}[1]{\left( #1 \right)}

\newcommand{\RR}{\mathbb{R}}

\usepackage{stmaryrd}
\usepackage{mathtools}
\DeclarePairedDelimiter{\norm}{\lVert}{\rVert}
\newcommand{\comment}[1]{}

\usepackage{hyperref}
\usepackage{cleveref}

\usepackage{hhline} 
\usepackage{multirow} 

\def\xHone{{\rm H}^{1}}
\def\xLn#1{{\rm L}^#1}
\def\xHonehalf{{\rm H}^{1/2}}
\def\Hzz{{\rm H}^{1/2}_{00}}
\newcommand{\xpdrv}[2]{{\frac{\partial #1}{\partial #2}}}
\newtheorem{thrm}{Theorem}

\title{Analysis of the Stokes--Darcy problem with generalised interface  conditions}
\author{Elissa         
    Eggenweiler\footnote{University of Stuttgart, Institute of Applied Analysis and Numerical Simulation, 
    emails: elissa.eggenweiler@ians.uni-stuttgart.de, rybak@ians.uni-stuttgart.de}, 
    Marco Discacciati\footnote{Loughborough University, Department of Mathematical Sciences, email: m.discacciati@lboro.ac.uk} and
	Iryna Rybak\footnotemark[1] 
	\\
	 \\
	}
\vspace{-100ex}
\date{ }
\begin{document}

\maketitle

\begin{abstract}
Fluid flows in coupled systems consisting of a free-flow region and the adjacent porous medium appear in a variety of environmental settings and industrial applications. In many applications, fluid flow is non-parallel to the fluid--porous interface that  requires a generalisation of the Beavers--Joseph coupling condition typically used for the Stokes--Darcy problem. Generalised coupling conditions valid for arbitrary flow directions to the interface are recently derived using the theory of homogenisation and boundary layers. The aim of this work is the mathematical analysis of the Stokes--Darcy problem with these  generalised interface conditions. We prove the existence and uniqueness of the weak solution of the coupled problem. The well-posedness is guaranteed under a suitable relationship between the permeability and the boundary layer constants containing geometrical information about the porous medium and the interface.
We numerically study the validity of the obtained results for realistic problems and provide a benchmark for numerical solution of the Stokes--Darcy problem with generalised interface conditions. 
\end{abstract}
\vspace{3ex}
\\[2ex]
\\[2ex]

\section*{Introduction}

Multi-domain flow systems containing a free-flow region and a porous medium with a common interface appear in a wide range of environmental settings and technical applications, e.g., soil-atmospheric interactions, industrial filtration and drying processes, water-gas management in fuel cells~\cite{Beaude_etal_19, Hanspal_etal_09, Jarauta_etal_20}. Mathematical models for such coupled flow systems convey the conservation of mass, momentum and energy, both in the two flow domains and across the fluid--porous interface. In the most general case, the Navier--Stokes equations are applied to describe fluid flow in the free-flow domain and multi-phase Darcy's law is used in the porous medium~\cite{Discacciati_Quarteroni_09, Mosthaf_Baber_etal_11}. However, depending on the application of interest and the flow regime, various simplifications of this general system are possible~\cite{Dawson_08, Maxwell2014, Reuter_etal_19, Rybak_etal_15, Sochala_Ern_Piperno_09}. 

The most widely studied free-flow and porous-medium flow system is described by the coupled Stokes--Darcy equations with different sets of interface conditions~\cite{Angot_etal_17, Discacciati_Miglio_Quarteroni_02, Goyeau_Lhuillier_etal_03, Jaeger_Mikelic_09, Bars_Worster_06, OchoaTapia_Whitaker_95}. Most of these coupling concepts are based on the Beavers--Joseph condition on the tangential velocity or its simplification by Saffman~\cite{Beavers_Joseph_67,  Discacciati_Quarteroni_09, Jaeger_Mikelic_09, Bars_Worster_06, Nield_09, Saffman_71}. 
However, these conditions are  developed for flows which are parallel to the fluid--porous interface, and therefore not applicable to arbitrary flow directions at the fluid--porous interface, e.g., for industrial filtration problems~\cite{Eggenweiler_Rybak_20}. In spite of the fact that they provide inaccurate results for arbitrary flow directions to the porous layer (Fig.~\ref{fig:advantages}),
they are still routinely used in the literature.

Alternative coupling conditions existing in the literature are either theoretically derived and involving unknown model parameters, which need to be calibrated before they can be used in computational models~\cite{Angot_18,Angot_etal_17}, or they are not justified for arbitrary flow directions at the fluid--porous interface~\cite{Lacis_Bagheri_17, Lacis_etal_20, Zampogna_Bottaro_16}. These limitations of the existing interface conditions severely restrict the variety of applications that can be accurately modelled. Recently, generalised interface conditions have been proposed in~\cite{Eggenweiler-MMS}. These conditions recover the classical conservation of mass and the balance of normal forces for isotropic porous media and provide an extension of the Beavers--Joseph condition. They reduce to those developed in~\cite{Jaeger_Mikelic_00, Jaeger_Mikelic_09} for parallel flows to the porous layer under the same assumptions on the flow direction, but they are valid for arbitrary flow directions to the porous layer (Fig.~\ref{fig:advantages}) and do not contain any unknown parameters. All the effective coefficients appearing in these generalised conditions are computed numerically based on the pore-scale geometrical information of the coupled flow system. The goal of this paper is to prove the well-posedness of the coupled Stokes--Darcy problem with this new set of interface conditions.  

The coupled Stokes--Darcy problem has been extensively studied in the last decade using the Beavers--Joseph--Saffman interface condition on the tangential component of the free-flow velocity~\cite{Discacciati_Quarteroni_09, GiraultRiviere,Kanschat_Riviere_10,Layton_Schieweck_Yotov_03}. Considering the original Beavers--Joseph condition, where the tangential component of the porous-medium velocity is not neglected, makes proving the well-posedness of the Stokes-Darcy problems quite challenging as it can be seen in~\cite{Cao_10,Hou_Qin_19}. This becomes even more difficult when the generalised conditions of \cite{Eggenweiler-MMS} are adopted. 


The paper is organised as follows. In Section~\ref{sec:models}, we provide the formulation of the coupled Stokes--Darcy problem with the generalised interface conditions and show the advantage of these conditions over the classical ones based on the Beavers--Joseph condition. In Section~\ref{sec:weak-form}, we derive the weak formulation of the coupled Stokes--Darcy problem with the generalised interface conditions and prove the existence and uniqueness of the weak solution. The well-posedness is guaranteed  
for isotropic porous media under a suitable relationship between the permeability and the boundary layer constants which contain the geometrical information about the interface. In Section~\ref{sec:numerics}, we study different porous-medium configurations and analyse the range of validity for the assumptions on the permeability and boundary layer constants. Then, we provide detailed information on how to compute the effective model parameters and present a benchmark for the Stokes--Darcy problem with generalised interface conditions including numerical simulation results. Concluding remarks and future work are presented in Section~\ref{sec:summary}.

\section{Coupled flow model}\label{sec:models}

In this paper, we consider the following assumptions on the coupled flow system. The flow domain $\Omega = \Omega_\FF \cup \Omega_\PM \subset \RR^2$ consists of the free-flow region $\Omega_\FF$ and the adjacent porous medium $\Omega_\PM$. The sharp interface $\Gamma$ separating the two flow regions at the macroscale is considered to be straight (Fig.~\ref{fig:domain}, left) and simple, i.e., mass, momentum and energy cannot be stored at or transported along $\Gamma$. We assume that the macroscale and the pore scale are separable, i.e., $\varepsilon=\ell/\mathcal{L} \ll 1$, where $\varepsilon$ is the scale separation parameter, $\ell$ is the characteristic pore length and $\mathcal{L}$ is the macroscopic length of the domain $\Omega$ (Fig.~\ref{fig:domain}, right).

The porous medium is considered to be non-deformable and homogeneous, constructed by a periodic repetition of solid obstacles. We consider the same single-phase and steady-state fluid flow at low Reynolds numbers both in the free-flow domain and through the porous medium. The fluid is supposed to be incompressible and to have constant viscosity. The coupled flow system is assumed to be isothermal.

\begin{figure}
    \centering
    \includegraphics[scale=1.1]{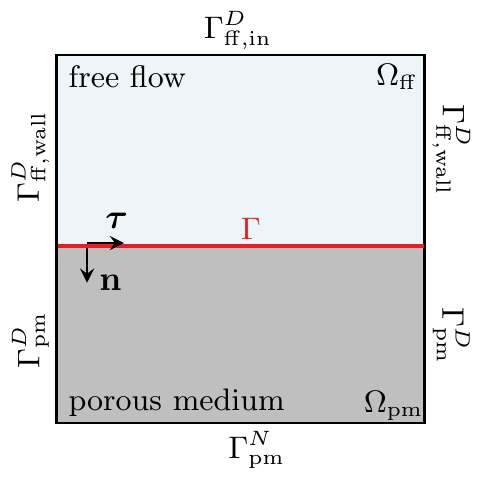} \hspace{11ex}
    \includegraphics[scale=1.1]{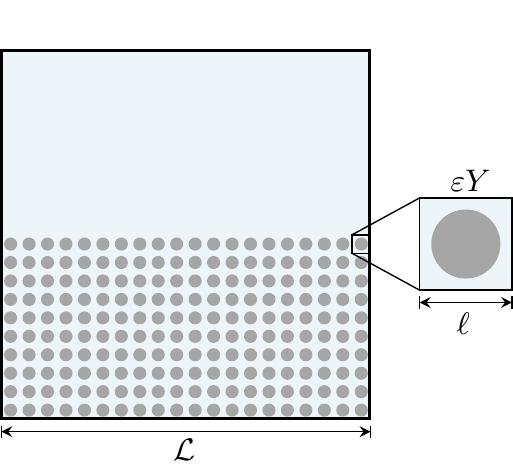} 
    \caption{Schematic coupled flow system at the macroscale (left) and at the pore scale (right) with the periodicity cell $\varepsilon Y$.}
    \label{fig:domain}
\end{figure}

\subsection{Free-flow model}

Under the given assumptions the  Stokes equations describe the fluid flow in the free-flow region
\begin{align}
\deld \vec{v}_\FF = 0 \qquad \text{in} \;\; \Omega_\FF, \label{eq:1p1c-FF-mass}
\\
- \deld \ten T(\vec v_\FF,p_\FF) = \vec 0
\qquad \text{in} \;\; \Omega_\FF, 
\label{eq:1p1c-FF-momentum}
\end{align}
where $\vec v_\FF$ is the fluid velocity,
$p_\FF$ is the fluid pressure, 
$\ten T(\vec v_\FF,p_\FF) = \nabla \vec v_\FF - p_\FF \ten I$ is the non-dimensional stress
tensor and $\ten I$ is the identity tensor. 

On the external boundary of the free-flow domain $\partial 
\Omega_\FF \setminus\Gamma$, the following Dirichlet boundary conditions are imposed
\begin{align}\label{eq:FF-BC}
\vec v_\FF = {\vec v}_\text{in} \hspace*{3.5mm} \text{on} \;\; \Gamma_{\FF,\text{in}}^D, 
\qquad \quad
\vec v_\FF = \vec 0 \hspace*{3.5mm} \text{on} \;\; \Gamma_{\FF,\text{wall}}^D, 
\end{align}
where $\displaystyle {\vec v}_\text{in}$ is the assigned velocity field (Fig.~\ref{fig:domain}, left).

\subsection{Porous-medium model}

In the porous-medium domain, we consider the non-dimensional Darcy flow equations  
\begin{align}\label{eq:PM-mass}
\deld \vec v_\PM &= 0  \hspace*{2.23cm} \text{in} \;\; 
\Omega_\PM, 
\\\label{eq:PM-Darcy}
\vec v_\PM &= -\ten K \del p_\PM  \hspace*{1.cm} \text{in} \;\; 
\Omega_\PM,
\end{align}
where $\vec v_\PM$ is the fluid velocity through the porous medium, $p_\PM$ is the fluid pressure and $\ten K$ is the intrinsic permeability tensor, which is symmetric 
positive definite and bounded. 

On the external boundary of the porous-medium domain $ \partial 
\Omega_\PM \setminus\Gamma$, we prescribe the following boundary conditions, that we consider homogeneous without loss of generality,
\begin{equation}\label{eq:PM-BC}
p_\PM = 0 
\quad \text{on} \;\; \Gamma_\PM^D, 
\qquad \quad
\vec v_\PM \vdot \vec n_\PM = 0 
\quad \text{on} \;\; \Gamma_ 
\PM^N,
\end{equation}
where $\vec n_\PM$ is the unit outward normal vector from the domain $\Omega_\PM$ on its boundary, $\partial 
\Omega_\PM \setminus\Gamma = 
\Gamma_\PM^D \cup \Gamma_\PM^N$, \, $\Gamma_\PM^D \cap \Gamma_\PM^N=\emptyset$, and $\Gamma_\PM^D \neq \emptyset$ (Fig.~\ref{fig:domain}, left).

\subsection{Interface conditions}\label{sec:interface-conditions}

The generalised interface conditions for the Stokes--Darcy problem~\eqref{eq:1p1c-FF-mass},~\eqref{eq:1p1c-FF-momentum} and \eqref{eq:PM-mass},~\eqref{eq:PM-Darcy} proposed in~\cite{Eggenweiler-MMS} 
read
\begin{align}\label{eq:IC-mass}
\vec v_\FF \vdot \vec n &= \vec v_\PM \vdot \vec n \hspace{6.cm} \text{on} \;\Gamma,
\\[1ex]
\label{eq:IC-momentum}
p_\PM &=
-\vec n \vdot \ten T\lrp{\vec v_\FF, p_\FF} \vec n 
-N_s^{bl} \, \vec \tau \vdot \ten T\lrp{\vec v_\FF, p_\FF} \vec n \hspace{1.325cm}  \text{on} \; \Gamma,
\\
\vec v_\FF  \vdot \vec \tau &=  \varepsilon N_\vec{\tau}^{bl}\,   \vec \tau \vdot \ten T\lrp{\vec v_\FF, p_\FF} \vec n  + \varepsilon^2 \sum_{j=1}^2 M_\vec{\tau}^{j,bl}
\xpdrv{p_\PM}{x_j} \hspace{1.cm}
 \text{on} \; \Gamma, \label{eq:IC-tangential}
\end{align}
where $N^{bl}_s$, $N_\vec{\tau}^{bl}= \vec N^{bl} \vdot \vec \tau$ and $M_\vec{\tau}^{j,bl}= \vec M^{j,bl} \vdot \vec \tau$ are boundary layer constants, $\vec n = -\vec n_\PM$ on $\Gamma$, and $\vec \tau$ is the unit tangential vector on $\Gamma$ (Fig.~\ref{fig:domain}, left). 

The interface condition~\eqref{eq:IC-mass} is the conservation of mass across the interface. The coupling condition~\eqref{eq:IC-momentum} is an extension of the balance of normal forces. In the case of isotropic porous media 
$N_s^{bl} =0$, that leads to the classical balance of normal forces at the fluid--porous interface, e.g.,~\cite{Discacciati_Miglio_Quarteroni_02, Layton_Schieweck_Yotov_03}. 
The interface condition~\eqref{eq:IC-tangential} is a generalisation of the Beavers--Joseph condition~\cite{Beavers_Joseph_67}:
\begin{align}
  \left(\vec v^\FF - \vec v^\PM\right) \vdot \vec \tau &= - 
   \frac{\sqrt{\ten K}}{\alpha_\BJ} \, \vec \tau \vdot \ten T\lrp{\vec v_\FF, p_\FF} \vec n \hspace{6.5ex} \text{on } \Gamma, \label{eq:IC-BJ}
\end{align}
where $\alpha_\BJ>0$ is the Beavers--Joseph parameter. Notice that the permeability tensor $\ten K$ is of order $\mathcal{O}(\varepsilon^2)$, i.e., $\ten K=\varepsilon^2 \tilde {\ten K}$, where $\tilde {\ten K}$ is computed in the standard way using  homogenisation theory. 
Moreover, $N_{\vec \tau}^{bl}<0$ and $M_{\vec \tau}^{1,bl}<0$. Condition~\eqref{eq:IC-tangential} can be compared to the Beavers--Joseph condition~\eqref{eq:IC-BJ} considering $-\varepsilon N_{\vec\tau}^{bl} \sim \sqrt{\ten K} \alpha_\BJ^{-1}$ and splitting the last term in equation~\eqref{eq:IC-tangential} into the tangential porous-medium velocity $\vec v^\PM \vdot \vec \tau$ and the remaining part, which is not necessarily zero in the generalised case. Note that for isotropic porous media $M_\vec{\tau}^{2,bl}=0$. 
For a thorough discussion and physical interpretation of the new interface condition \eqref{eq:IC-tangential} we refer the reader to \cite{Eggenweiler-MMS}.

The generalised coupling conditions~\eqref{eq:IC-mass}--\eqref{eq:IC-tangential} have several advantages over the classical conditions (conservation of mass, balance of normal forces, Beavers--Joseph condition). In Figure~\ref{fig:advantages}, we present tangential velocity profiles for a setting similar to~\cite{Eggenweiler-MMS}, where the flow is arbitrary (Fig.~\ref{fig:advantages}, left) and parallel (Fig.~\ref{fig:advantages}, right) to the fluid--porous interface. Macroscale Stokes--Darcy problems with the classical (profile:~classical IC) and generalised (profile:~generalised IC) interface conditions are compared against the pore-scale resolved simulations (profile:~pore-scale). One can observe that the generalised conditions are suitable for arbitrary flow directions to the interface (Fig.~\ref{fig:advantages}, left) and are more accurate than the classical conditions for parallel flows  (Fig.~\ref{fig:advantages}, right). Moreover, the interface conditions~\eqref{eq:IC-mass}--\eqref{eq:IC-tangential} contain no undetermined parameters such as $\alpha_\BJ$, which needs to be fitted. For more details on the validation of the generalised interface conditions and their comparison to the classical conditions we refer the reader  to~\cite{Eggenweiler-MMS}.

\begin{figure}
    \centering
    \includegraphics[scale=0.9]{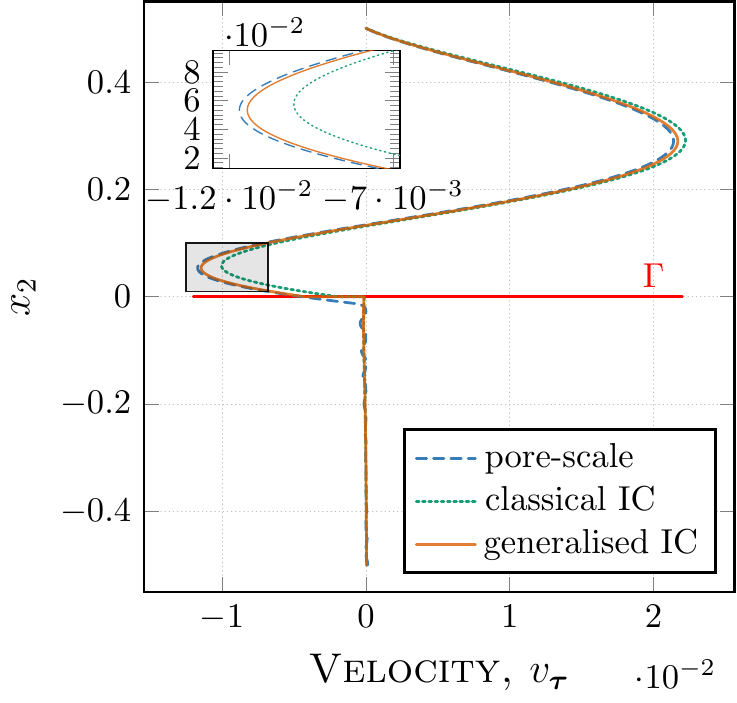}
    \quad \quad  \includegraphics[scale=0.9]{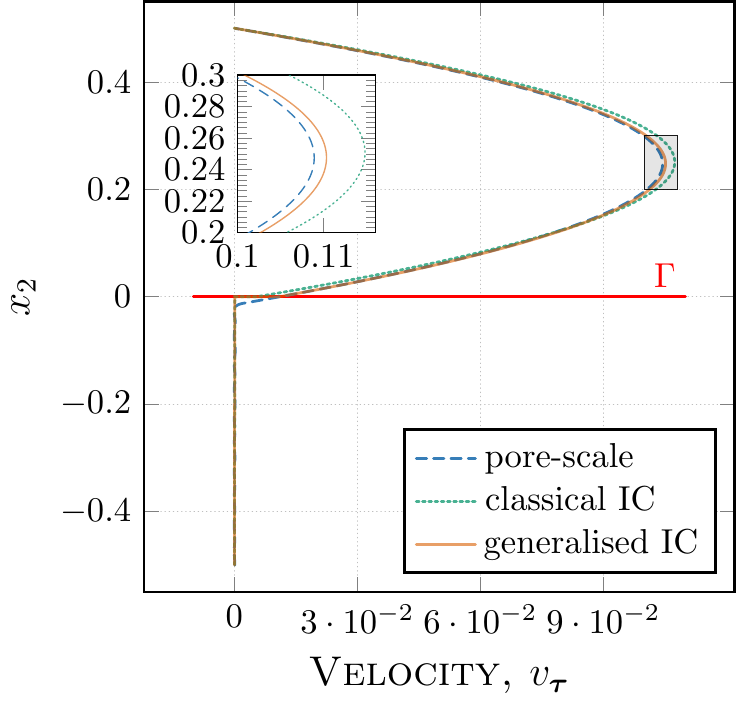}
    \caption{Comparison of the generalised and the classical interface conditions for arbitrary (left) and parallel (right) flows to the interface based on the flow problem  from~\cite{Eggenweiler-MMS}. }
    \label{fig:advantages}
\end{figure}

\section{Weak formulation and analysis}
\label{sec:weak-form}
In this section, we study the weak formulation of the coupled Stokes--Darcy problem \eqref{eq:1p1c-FF-mass}--\eqref{eq:IC-tangential} and analyse its well-posedness.

\subsection{Weak formulation of the Stokes--Darcy problem}


We introduce the following functional spaces 
\begin{align*}
    \mathrm{H}_\FF &:= \left\{ \vec u \in \xHone(\Omega_\FF)^2: \vec u = \vec 0 \text{ on } \partial\Omega_{\FF} \setminus \Gamma \right\},
    \\
    \mathrm{H}_{\FF,\Gamma} &:= \left\{ \vec u \in \xHone(\Omega_\FF)^2: \vec u = \vec 0 \text{ on } \Gamma_{\FF,\text{wall}}^D  \cup \Gamma \right\}, 
    \ \
    \mathrm{Q} := \xLn{2}(\Omega_\FF),
    \\
    \mathrm{H}_\PM &:= \left\{ \psi \in \xHone(\Omega_\PM): \psi = 0 \text{ on } \Gamma_{\PM}^D \right\},
    \ \
     \mathrm{W} :=\mathrm{H}_\FF \times \mathrm{H}_\PM,
\end{align*}
and norms
\begin{align*}
    & \norm{\varphi}_{0,i} := \norm{\varphi}_{\xLn{2}(\Omega_i)}  \quad \forall \varphi  \in \xLn{2}(\Omega_i),
    \quad \qquad 
    \norm{\varphi}_{1,i} := \norm{\varphi}_{\xHone(\Omega_i)} \quad  \forall \varphi  \in \xHone(\Omega_i),
    \\
    & \norm{\underline w}_W := \left(\norm{\vec w}_{1,\FF}^2 +  \norm{\psi}_{1, \PM}^2 \right)^{1/2} \quad \forall \underline w = (\vec w, \psi) \in \mathrm{W},
\end{align*}
where the subscript $i\in \{\FF,\PM\}$ indicates in which domain a function is defined. Analogous notations apply for vector-valued functions.
Additionally, on the interface $\Gamma$ we consider  the trace space $H_{00}^{1/2}(\Gamma)$ and denote its dual space by $(H_{00}^{1/2}(\Gamma))'$, see, e.g.,~\cite{Lions_68}.
In the following, for the sake of simplicity, we waive $\text d \vec x$ and $\text{d} S$ in volume and boundary integrals and write $g =g(\vec x)$ for $\vec x \in \Omega_i$, $i\in \{\FF,\PM\}$.

In order to obtain the weak formulation of the Stokes--Darcy problem~\eqref{eq:1p1c-FF-mass}--\eqref{eq:IC-tangential}, we multiply equation~\eqref{eq:1p1c-FF-momentum} by a test function $\vec u \in \mathrm{H}_\FF$ and proceeding in a standard way, we obtain
\begin{align}
    0 &= -\int_{\Omega_\FF} \!\! \left( \deld \ten T\left(\vec v_\FF,p_\FF\right) \right)\vdot \vec u 
    = 
    -\int_{\partial \Omega_\FF}  \!\! \ten T\lrp{\vec v_\FF, p_\FF} \vec n \vdot \vec u  
    +\int_{\Omega_\FF}\!\! \nabla \vec v_\FF \colon \nabla \vec u
    -\int_{\Omega_\FF} \!\!  p_\FF \deld \vec u \notag
    \\
    &=  -\int_{\Gamma} \ten T\lrp{\vec v_\FF, p_\FF} \vec n \vdot \vec u
    +\int_{\Omega_\FF} \nabla \vec v_\FF \colon \nabla \vec u
    -\int_{\Omega_\FF}  p_\FF \deld \vec u \  \qquad \forall \vec u \in \mathrm{H}_\FF \ .
    \label{eq:weak-FF-first}
\end{align}


Splitting the stress tensor $\ten T(\vec v_\FF, p_\FF)$ in its normal and tangential components and applying the interface conditions~\eqref{eq:IC-momentum} and~\eqref{eq:IC-tangential} for the integral term over $\Gamma$ in~\eqref{eq:weak-FF-first}, we get
\begin{align}
    \int_{\Gamma}   \ten T\lrp{\vec v_\FF, p_\FF}\vec n  \vdot \vec u
    &= \int_{\Gamma} \left( \vec n \vdot  \ten T\lrp{\vec v_\FF, p_\FF}\vec n \right) ( \vec u \vdot \vec n)
    +\int_{\Gamma}  \left( \vec \tau \vdot \ten T\lrp{\vec v_\FF, p_\FF}\vec n \right) ( \vec u \vdot \vec \tau)
    \notag
    \\
    &
    =\int_{\Gamma} \left(- p_\PM - N_s^{bl} \vec \tau \vdot \ten T\lrp{\vec v_\FF, p_\FF} \vec n  \right) (\vec u \vdot \vec n)
    +\int_{\Gamma}  \left( \vec \tau \vdot \ten T\lrp{\vec v_\FF, p_\FF}\vec n \right) (\vec u \vdot \vec \tau)
    \notag
    \\
    &=-\int_{\Gamma} p_\PM (\vec u \vdot \vec n)
    + \int_{\Gamma} \left(  \vec \tau \vdot \ten T\lrp{\vec v_\FF, p_\FF} \vec n  \right) \left( (\vec u \vdot \vec \tau) -N_s^{bl} ( \vec u \vdot \vec n) \right)
    \notag
    \\
    &
    =
    - \int_{\Gamma} p_\PM (\vec u \vdot \vec n ) \notag
     + \int_{\Gamma} \left(  {(N_{\vec \tau}^{bl})}^{-1} \varepsilon^{-1} (\vec v_\FF  \vdot \vec \tau) - {(N_{\vec \tau}^{bl})}^{-1} \varepsilon \sum_{j=1}^2 M_{\vec \tau}^{j,bl}  \xpdrv{p_\PM}{x_j}  \right) \left((\vec u \vdot \vec \tau) -N_s^{bl} (\vec u \vdot \vec n)
     \right)
    \notag
    \\
    &= - \int_{\Gamma} p_\PM (\vec u \vdot \vec n)
    - \int_{\Gamma} {(N_{\vec \tau}^{bl})}^{-1} N_s^{bl} \, \varepsilon^{-1}  (\vec v_\FF  \vdot \vec \tau ) \left(\vec u \vdot \vec n
    \right)
    + \int_{\Gamma}  {(N_{\vec \tau}^{bl})}^{-1} \varepsilon^{-1}  (\vec v_\FF  \vdot \vec \tau) \left( \vec u \vdot \vec \tau \right)
    \notag
    \\
    & \quad + \int_{\Gamma} \left(  {(N_{\vec \tau}^{bl})}^{-1} N_s^{bl} \, \varepsilon \sum_{j=1}^2 M_{\vec \tau}^{j,bl}  \xpdrv{p_\PM}{x_j}   \right)  (\vec u \vdot \vec n)
    - \int_{\Gamma} \left(  {(N_{\vec \tau}^{bl})}^{-1} \varepsilon \sum_{j=1}^2 M_{\vec \tau}^{j,bl}  \xpdrv{p_\PM}{x_j}  \right) \left(\vec u \vdot \vec \tau \right) \ .
    \label{eq:stress-on-interface}
\end{align}
We introduce a continuous lifting operator $E_\FF: \xHonehalf(\Gamma_{\FF,\text{in}}^D)^2 \rightarrow  \mathrm{H}_{\FF,\Gamma}$ and split the free-flow velocity $\vec v_\FF = \vec v_\FF^0 + E_\FF\vec v_\text{in}$ with $\vec v_\FF^0 \in \mathrm{H}_\FF$. 
Substituting~\eqref{eq:stress-on-interface} into the weak formulation~\eqref{eq:weak-FF-first}, we get
\begin{align}
    -\int_{\Omega_\FF} \nabla \left( E_\FF\vec v_\text{in} \right) \colon \nabla \vec u
    &= 
    \int_{\Omega_\FF} \nabla \vec v_\FF^0 \colon \nabla \vec u
    -\int_{\Omega_\FF}  p_\FF (\deld \vec u)
    + \int_{\Gamma} p_\PM (\vec u \vdot \vec n) 
    + \int_{\Gamma} {(N_{\vec \tau}^{bl})}^{-1} N_s^{bl} \, \varepsilon^{-1} (\vec v_\FF^0  \vdot \vec \tau)  \left(\vec u \vdot \vec n
    \right) 
    \notag
    \\
    & \quad 
    - \int_{\Gamma}  {(N_{\vec \tau}^{bl})}^{-1} \varepsilon^{-1} (\vec v_\FF^0  \vdot \vec \tau) \left( \vec u \vdot \vec \tau \right) 
    - \int_{\Gamma} \left(  {(N_{\vec \tau}^{bl})}^{-1} N_s^{bl} \, \varepsilon \sum_{j=1}^2 M_{\vec \tau}^{j,bl}  \xpdrv{p_\PM}{x_j} \right) (\vec u \vdot \vec n)
    \notag
    \\
    & \quad 
    + \int_{\Gamma} \left(  {(N_{\vec \tau}^{bl})}^{-1} \varepsilon \sum_{j=1}^2 M_{\vec \tau}^{j,bl}  \xpdrv{p_\PM}{x_j} \right) \left(\vec u \vdot \vec \tau \right), \qquad \forall \vec u \in \mathrm{H}_\FF \, .
    \label{eq:weak-FF-a}
\end{align}
The weak form of equation~\eqref{eq:1p1c-FF-mass} reads
\begin{align}
    - \int_{\Omega_\FF} \left( \nabla \vdot \vec v_\FF^0 \right) q = \int_{\Omega_\FF} \left( \nabla \vdot (E_\FF\vec  v_\text{in} )\right)q \qquad \forall q \in \mathrm{Q} \, .
\end{align}


In the porous medium, we consider the scalar elliptic formulation obtained by combining equation~\eqref{eq:PM-mass} with Darcy's law~\eqref{eq:PM-Darcy}. Multiplying the resulting equation by a test function $\psi \in \mathrm{H}_\PM$, integrating over $\Omega_\PM$ and applying the interface condition~\eqref{eq:IC-mass}, we get
\begin{align}
      \int_{\Omega_\PM} \left( \ten K \nabla p_\PM \right) \vdot \nabla \psi
      -\int_{\Gamma} \left(  \vec v_\FF \vdot  \vec n \right) \psi = 0  \qquad \forall \psi \in \mathrm{H}_\PM.
\end{align}

For all $\underline v = (\vec v, \varphi), \; \underline w = (\vec w, \psi) \in \mathrm{W}
$ and $q \in \mathrm{Q}$, we define the following bilinear forms
\begin{align}
    \mathcal{A}(\underline v,\underline w) =& \int_{\Omega_\FF} \nabla \vec v \colon \nabla \vec w 
    + \int_{\Omega_\PM} (\ten K \nabla \varphi) \vdot  \nabla \psi   
    +\int_{\Gamma} \varphi ( \vec w \vdot \vec n)
    - \int_{\Gamma} (\vec v \vdot \vec n) \psi
    + \int_{\Gamma} {(N_{\vec \tau}^{bl})}^{-1} N_s^{bl} \, \varepsilon^{-1} (\vec v  \vdot \vec \tau ) \left(\vec w \vdot \vec n  \right) \notag
    \\
    & 
    - \int_{\Gamma}  {(N_{\vec \tau}^{bl})}^{-1} \varepsilon^{-1} (\vec v  \vdot \vec \tau ) \left( \vec w \vdot \vec \tau \right) \notag
    - \int_{\Gamma} \left(  {(N_{\vec \tau}^{bl})}^{-1} N_s^{bl}  \, \varepsilon \sum_{j=1}^2 M_{\vec \tau}^{j,bl} \xpdrv{\varphi}{x_j}   \right) ( \vec w \vdot \vec n)\notag
    \\
    & + \int_{\Gamma} \left(  {(N_{\vec \tau}^{bl})}^{-1} \varepsilon \sum_{j=1}^2 M_{\vec \tau}^{j,bl}  \xpdrv{\varphi}{x_j}  \right) \left(\vec w \vdot \vec \tau \right) \, ,
    \label{eq:definitionA}  
    \\
    \mathcal{B}(\underline w, q) =& -\int_{\Omega_\FF} (\nabla \vdot \vec w) q \,  ,
    \label{eq:definitionB}
\end{align}
and the  linear functionals
\begin{align}
    \mathcal{F}(\underline w) =& -\int_{\Omega_\FF} \nabla \left(E_\FF \vec v_\text{in}\right) \colon \nabla \vec w \, ,
    \qquad \quad
    \mathcal{G}(q) = \int_{\Omega_\FF} (\nabla \vdot \left( E_\FF \vec v_\text{in} \right)) q \, .
    \label{eq:FG}
\end{align}

Making use of these notations, the weak formulation of the coupled Stokes--Darcy problem~\eqref{eq:1p1c-FF-mass}--\eqref{eq:IC-tangential} reads:

find $\underline{u} = 
(\vec v^0_\FF, p_\PM)\in \mathrm{W}$ and $ p_\FF \in \mathrm{Q}$, such that
\begin{align}
    \mathcal{A}(\underline u, \underline w) + \mathcal{B}(\underline w, p_\FF) &=  \mathcal{F}(\underline w) \qquad \forall \underline w = (\vec w, \psi) \in \mathrm{W} \, 
    , \label{weak-form-1}
    \\
     \mathcal{B}(\underline u, q) &= \mathcal{G}(q)  \qquad \ \ \forall q \in \mathrm{Q} \, .
     \label{weak-form-2}
\end{align}


\subsection{Analysis of the Stokes--Darcy model}
\label{sec:analysis}

In this section, we prove the well-posedness of the Stokes--Darcy problem with the generalised interface conditions~\eqref{eq:IC-mass}--\eqref{eq:IC-tangential}. 
Since we consider the interface $\Gamma$ to be straight (see Sect.~\ref{sec:models}), the tangential and normal vectors at the interface are constant. Therefore, taking the tangential and normal components $ \vec v \vdot \vec \tau$ and $\vec v \vdot  \vec n$  of a suitable vector function $\vec v $ on $\Gamma$ does not reduce the regularity of the trace $\vec v|_\Gamma$. 
In the following, we present the results for a horizontal interface $\Gamma$, so that $\vec \tau= \vec e_1$ and $\vec n = -\vec e_2$ (Fig.~\ref{fig:domain}, left).

We prove the well-posedness for isotropic porous media, i.e., $\ten K = k \ten I $ with $k>0$ constant and equal to $k = \varepsilon^2 \tilde k$, where $\tilde k$ is the non-dimensional permeability (see Section~\ref{sec:interface-conditions}). Note that in this case the boundary layer constants $N_s^{bl}=0$ and $M_{\vec \tau}^{2,bl}=0$. Thus, the bilinear form $\mathcal{A}(\underline u, \underline w)$ in~\eqref{weak-form-1} reduces to

\begin{align}
    \mathcal{A}(\underline u,\underline w) =& \int_{\Omega_\FF} \nabla \vec v_\FF^0 \colon \nabla \vec w 
    + \int_{\Omega_\PM} (k \nabla p_\PM) \vdot  \nabla \psi     
    +\int_{\Gamma} p_\PM(\vec w \vdot \vec n)
    - \int_{\Gamma} (\vec v_\FF^0 \vdot \vec n) \psi
    - \int_{\Gamma}  {(N_{\vec \tau}^{bl})}^{-1} \varepsilon^{-1} ( \vec v_\FF^0  \vdot \vec \tau) \left( \vec w \vdot \vec \tau \right) \notag
    \\
    & + \int_{\Gamma} {(N_{\vec \tau}^{bl})}^{-1} \varepsilon \left(   
    M_{\vec \tau}^{1,bl} 
    \xpdrv{p_\PM}{x_1}  \right) \left(\vec w \vdot \vec \tau \right) \, .
    \label{eq:definitionA-isotropic}
\end{align}
Section ~\ref{sec:auxiliary-results} introduces some auxiliary results, while the well-posedness of the coupled Stokes--Darcy problem is proved in Section~\ref{sec:well-posedness}. 

\subsubsection{Auxiliary results}
\label{sec:auxiliary-results}

Taking into account the Poincar\'{e} inequality
\begin{align}
    \exists C_{P,i}>0 \quad \text{such that } \norm{f}_{0,i} \leq C_{P,i} \norm{\nabla f}_{0, i}  \qquad \forall f \in \mathrm{H}_i \,,
    \label{eq:Poincare}
\end{align}
and the definition of the $\xHone$-norm 
$ \norm{f}_{1,i}^2 = \norm{f}_{0,i}^2 + \norm{\nabla f}_{0,i}^2$ for $i\in \{\FF,\PM\}$,
we get 
\begin{align}
    \norm{f}_{1,i}^2 \leq \kappa_i \norm{\nabla f}_ {0,i}^2  \qquad \forall f \in \mathrm{H}_i \,,
    \label{eq:kappa}
\end{align}
where the constant $\kappa_i = 1+ {C_{P,i}}^2 >1$.

We consider the following trace inequalities (see, e.g.,~\cite{Lions_68}):
\begin{align}
    &\exists C_f>0 \quad \text{such that } \norm{\vec v|_\Gamma}_{\Hzz(\Gamma)} \leq C_f \norm{\vec v}_{1, \FF} \qquad \forall \vec v \in \mathrm{H}_\FF \, , \label{eq:trace-ff}
    \\
    &\exists C_p>0 \quad \text{such that } \norm{\psi|_\Gamma}_{\Hzz(\Gamma)} \leq C_p \norm{\psi}_{1,\PM} \qquad \forall f \in \mathrm{H}_\PM \; . \label{eq:trace-pm}
\end{align}

In the two-dimensional case, for $\varphi \in \xHone(\Omega_\PM)$ we have
\begin{align*}
    \mathbf{curl} \, \varphi = \left( \xpdrv{\varphi}{x_2}, - \xpdrv{\varphi}{x_1}\right)^\top \ .
\end{align*}
Since $\nabla \varphi \in \xLn{2}(\Omega_\PM)$, we get  $\mathbf{curl} \, \varphi \in \xLn{2}(\Omega_\PM)$ and being $\nabla \vdot (\mathbf{curl} \, \varphi) = 0$, we have $\mathbf{curl} \, \varphi \in \mathrm{H}(\operatorname{div};\Omega_\PM):=\{ \vec u \in \xLn{2}(\Omega_\PM)^2 \,:\, \nabla \vdot \vec u \in \xLn{2}(\Omega_\PM)\}$. 
Thus, due to the classical trace results in $\mathrm{H}(\mbox{div};\Omega_\PM)$, see e.g., ~\cite[Lemma 20.2]{Tartar}, \cite[Chapter IX, Thm. 1]{Dautray_90}, there exists a positive constant $C_{\vec \tau} >0$ such that
\begin{equation}\label{eq:inequality_step1}
    \norm{ \mathbf{curl}\,\varphi\vdot \vec n }_{\mathrm{H}^{-1/2}(\partial\Omega_\PM)} 
    \leq C_{\vec \tau} \norm{ \mathbf{curl}\, \varphi }_{\mathrm{H}(\operatorname{div};\Omega_\PM)} = C_{\vec \tau} \norm{ \mathbf{curl}\,\varphi }_{0,\PM} 
    = C_{\vec \tau} \norm{ \nabla \varphi }_{0,\PM} \leq C_{\vec \tau} \norm{ \varphi }_{1,\PM} \, .
\end{equation}
Since $\Gamma \subsetneq \partial \Omega_\PM$, we have $(\mathbf{curl}\,\varphi\vdot \vec n)_{\vert\Gamma} \in (\Hzz(\Gamma))'$ with $\Hzz (\Gamma) = \{ u \in \xHonehalf(\partial \Omega_\PM) \, : \, \mbox{supp}\, u \subset \overline{\Gamma} \}$ and
\begin{equation}\label{eq:inequality_step2}
    \norm{ (\mathbf{curl}\,\varphi\vdot \vec n)_{\vert\Gamma} }_{(\Hzz(\Gamma))'} \leq \norm{ \mathbf{curl}\,\varphi\vdot \vec n }_{\mathrm{H}^{-1/2}(\partial \Omega_\PM)} \, .
\end{equation}
Therefore, as obviously $\nabla \varphi \cdot \vec \tau = \mathbf{curl} \, \varphi \cdot \vec n$, from \eqref{eq:inequality_step1} and \eqref{eq:inequality_step2} we conclude that
\begin{equation}
    \exists C_{\vec \tau} > 0 \quad \text{such that }
    \norm{(\nabla \varphi\vdot \vec \tau)_{\vert\Gamma}}_{(\Hzz(\Gamma))'} \leq C_{\vec \tau} \norm{ \varphi }_{1,\PM} \qquad \forall \varphi \in \xHone(\Omega_\PM)\, .
    \label{eq:inequality-tangential}
\end{equation}
Let us also remark that, considering the geometrical setting provided in Figure~\ref{fig:domain} (left) where $\vec \tau = \vec e_1$, we have
\begin{equation*}
\nabla p_\PM \vdot \vec \tau = \xpdrv{p_\PM}{x_1} \quad \mbox{on } \Gamma \, .
\end{equation*}
Finally, notice that the first three integrals over $\Gamma$ in the bilinear form $\mathcal{A}$ given in~\eqref{eq:definitionA-isotropic} should be understood as scalar products in $\Hzz(\Gamma)$, while the last integral must be interpreted as the duality pairing
\begin{align}
    &\int_{\Gamma}  ({N_{\vec \tau}^{bl}})^{-1} \varepsilon \left(   
    \! M_{\vec \tau}^{1,bl}
    \xpdrv{p_\PM}{x_1} \!  \right) \left(\vec w \vdot \vec \tau \right)
    =
    \int_{\Gamma} \varepsilon \frac{M_{\vec\tau}^{1,bl}}{N_{\vec\tau}^{bl}}   (\nabla p_\PM \vdot \vec \tau) (\vec w \vdot \vec \tau ) = \varepsilon 
    \frac{M_{\vec\tau}^{1,bl}}{N_{\vec\tau}^{bl}}  
    \langle \nabla p_\PM \vdot \vec \tau, \vec w \vdot \vec \tau 
    \rangle_{(\Hzz(\Gamma))',\Hzz(\Gamma)} \, .
    \label{eq:duality}
\end{align}

\subsubsection{Well-posedness of the coupled problem}
\label{sec:well-posedness}

Using the auxiliary results from Section~\ref{sec:auxiliary-results}, we can now prove the well-posedness of the Stokes--Darcy problem with the generalised interface conditions.

\begin{thrm}
The Stokes--Darcy problem~\eqref{weak-form-1},~\eqref{weak-form-2} is well-posed under the following assumption
\begin{equation}
    \tilde k  >  \kappa_\PM \kappa_\FF (C_{\vec \tau} C_{f} )^2 \left( \frac{M_{\vec \tau}^{1,bl}}{2 N_{\vec \tau}^{bl}}\right)^2 \, .
    \label{eq:coercivity-condition-result}
\end{equation}

\end{thrm}

\begin{proof} The proof uses the classical Babu\v ska-Brezzi theory for the well-posedness of saddle-point problems \cite{Brezzi:1974:OEU}. The continuity of the functionals $\mathcal{F}$ and $\mathcal{G}$ given by~\eqref{eq:FG} is straightforward as well as the continuity and coercivity of the bilinear form $\mathcal{B}$ given by~\eqref{eq:definitionB}. Thus, we focus only on the continuity and coercivity of the bilinear form $\mathcal{A}$ presented in~\eqref{eq:definitionA-isotropic}, starting from the former property.

Using the Cauchy--Schwarz inequality, the Poincar\'{e} inequality~\eqref{eq:Poincare}, the trace inequalities~\eqref{eq:trace-ff}, \eqref{eq:trace-pm}, \eqref{eq:inequality-tangential}, and result \eqref{eq:duality}, we get
\begin{align}
    |\mathcal{A}(\underline u,\underline w)| =& \bigg| \int_{\Omega_\FF} \nabla \vec v_\FF^0 \colon \nabla \vec w 
    + \int_{\Omega_\PM} (k \nabla p_\PM) \vdot  \nabla \psi     
    +\int_{\Gamma} p_\PM(\vec w \vdot \vec n)
    - \int_{\Gamma} (\vec v_\FF^0 \vdot \vec n) \psi
    - \int_{\Gamma}  {(N_{\vec \tau}^{bl})}^{-1} \varepsilon^{-1} ( \vec v_\FF^0  \vdot \vec \tau) \left( \vec w \vdot \vec \tau \right) \notag
    \\
    & + \int_{\Gamma}  {(N_{\vec\tau}^{bl})}^{-1} \varepsilon \left(  
    M_{\vec \tau}^{1,bl}
    \xpdrv{p_\PM}{x_1}  \right) \left(\vec w \vdot \vec \tau \right) \notag \bigg| 
    \notag
    \\
    \leq &\norm{\vec v_\FF^0}_{1,\FF}  \norm{ \vec w}_{1,\FF} +  k \norm{ p_\PM}_{1,\PM} \norm{\psi}_{1,\PM} + C_{f}C_{p} \norm{p_\PM}_{1,\PM} \norm{\vec w}_{1,\FF} + C_{f}C_{p}  \norm{\vec v_\FF^0}_{1,\FF} \norm{\psi}_{1,\PM}
    \notag
    \\
    & 
    + \varepsilon^{-1} C_{f}^2 \frac{1}{|N_{\vec\tau}^{bl}|}  \norm{\vec v_\FF^0}_{1,\FF} \norm{\vec w}_{1,\FF}
    + \varepsilon C_{\vec \tau} C_{f}\left|\frac{M_{\vec\tau}^{1,bl}}{N_{\vec\tau}^{bl}}\right|  \norm{p_\PM}_{1,\PM} \norm{\vec w}_{1,\FF} \ .
\end{align}
We define 
\begin{align*}
    \gamma := \operatorname{max} \left\{ k, \; 
    1 + 
    \varepsilon^{-1} C_{f}^2 \frac{1}{|N_{\vec\tau}^{bl}|}, \; C_{f} C_{p} + \varepsilon C_{\vec \tau} C_{f}\left|\frac{M_{\vec\tau}^{1,bl}}{N_{\vec\tau}^{bl}}\right| 
    \right\}
\end{align*}
and obtain
\begin{align}
    |\mathcal{A}(\underline u,\underline w)| 
    &\leq  \gamma \left( \norm{\vec v_\FF^0}_{1,\FF} + \norm{ p_\PM}_{1,\PM}\right)
    \left( \norm{\vec w}_{1,\FF} + \norm{\psi}_{1,\PM}\right)
    \leq 2 \gamma \norm{\underline{u}}_W \norm{\underline w}_W,
    \label{eq:continuity}
\end{align}
where the second inequality in equation~\eqref{eq:continuity} follows from $(a+b) \leq \sqrt{2}(a^2+b^2)^{1/2}$ for all $a,b \in \mathbb{R}^+$. Thus,  the bilinear form $\mathcal{A}$ is continuous. 

Now we prove the coercivity of $\mathcal{A}$. We remember that $N_{\vec \tau}^{bl}<0$ and $M_{\vec \tau}^{1,bl}<0$ (see Sect.~\ref{sec:interface-conditions}). Making use of~\eqref{eq:kappa}, the trace inequality~\eqref{eq:trace-ff}, and \eqref{eq:duality}, we obtain
\begin{align}
    \mathcal{A}(\underline u,\underline u) =& 
    \int_{\Omega_\FF} \nabla \vec v_\FF^0 \colon \nabla \vec v_\FF^0
    + \int_{\Omega_\PM} (k \nabla p_\PM) \vdot  \nabla p_\PM   
    \underbrace{+\int_{\Gamma} p_\PM (\vec v_\FF^0 \vdot \vec n)
    - \int_{\Gamma} (\vec v_\FF^0 \vdot \vec n) p_\PM}_{=0}  
     \underbrace{- \int_{\Gamma}  {(N_{\vec\tau}^{bl})}^{-1} \varepsilon^{-1} \left( \vec v_\FF^0 \vdot \vec \tau \right)^2 }_{\geq 0} 
     \notag
    \\
    &
   \notag
    + \int_{\Gamma}  {(N_{\vec\tau}^{bl})}^{-1} \varepsilon \left( M_{\vec \tau}^{1,bl} \xpdrv{p_\PM}{x_1}  \right) \left(\vec v_\FF^0 \vdot \vec \tau \right) \notag 
    \notag
    \\
    \geq & 
    \norm{\nabla \vec v_\FF^0 }_{0,\FF}^2 
    + k \norm{\nabla p_\PM}_{0,\PM}^2 
    -  \varepsilon
    \frac{M_{\vec\tau}^{1,bl}}{N_{\vec\tau}^{bl}} \, \norm{ (\nabla p_\PM \vdot \vec \tau)|_\Gamma}_{(\Hzz(\Gamma))'} \, \norm{ (\vec v_\FF^0 \vdot \vec \tau)|_\Gamma }_{\Hzz(\Gamma)}  \notag 
    \\
    \geq  &
    {\kappa_\FF}^{-1} \norm{\vec v_\FF^0}_{1,\FF}^2
    + \kappa_\PM^{-1} k  \norm{ p_\PM}_{1,\PM}^2
     - \varepsilon C_{\vec \tau} C_{f} \frac{M_{\vec\tau}^{1,bl}}{N_{\vec\tau}^{bl}} \norm{ p_\PM}_{1,\PM} \, \norm{ \vec v_\FF^0}_{1,\FF}  \ .
    \label{eq:coercivity}
\end{align}
Applying the generalised Young's inequality $ab \leq a^2/(2\delta) + \delta b^2/2$ with  
$a=\varepsilon C_{\vec \tau} C_{f} {(N_{\vec \tau}^{bl})^{-1}} M_\tau^{1,bl}\norm{ p_\PM}_1 \geq 0$, $b= \norm{ \vec v_\FF^0}_1 \geq 0$ and $\delta > 0$ to the last term in~\eqref{eq:coercivity}, we get
\begin{align*} 
    \mathcal{A}(\underline u,\underline u) \geq \ &
    {\kappa_\FF}^{-1}
    \norm{\vec v_\FF^0}_{1,\FF}^2
    + \kappa_\PM^{-1} k  \norm{ p_\PM}_{1,\PM}^2
    - \frac{(\varepsilon C_{\vec \tau} C_{f} M_{\vec \tau}^{1,bl}{(N_{\vec \tau}^{bl})^{-1}})^2}{2\delta }\norm{ p_\PM}_{1,\PM}^2 - \frac{\delta }{2} \norm{ \vec v_\FF^0}_{1,\FF}^2 \notag
    \\
    =& \
    \left({\kappa_\FF}^{-1} 
    - \frac{\delta }{2} \right)\norm{\vec v_\FF^0}_{1,\FF}^2  \notag
    + \left( \kappa_\PM^{-1} k - \frac{(\varepsilon C_{\vec \tau} C_{f} M_{\vec \tau}^{1,bl}{(N_{\vec \tau}^{bl})^{-1}})^2}{2\delta} \right) \norm{ p_\PM}_{1,\PM}^2 \, .
\end{align*}
Recalling that the permeability can be written as $ k = \varepsilon^2 \tilde k $, we can conclude that the coercivity of the bilinear form $\mathcal{A}$ is guaranteed when the following conditions hold
\begin{align}
    \delta < 2{\kappa_\FF}^{-1} \, ,  \qquad 
     \kappa_\PM^{-1} \, \tilde k - \frac{(C_{\vec \tau} C_{f} M_{\vec \tau}^{1,bl}{(N_{\vec \tau}^{bl})^{-1}})^2}{2\delta}  > 0 \, . \label{eq:condition-proof}
\end{align}
Combining both inequalities in~\eqref{eq:condition-proof}, we obtain condition~\eqref{eq:coercivity-condition-result}. Therefore, the coercivity of the bilinear form $\mathcal{A}$ is guaranteed if condition~\eqref{eq:coercivity-condition-result} is fulfilled.
\end{proof}
The well-posedness of the Stokes--Darcy problem  with generalised coupling conditions is established under assumption~\eqref{eq:coercivity-condition-result} on the permeability being not too small in comparison with the ratio of boundary layer constants $(M_{\vec \tau}^{1,bl} / N_{\vec \tau}^{bl})^2 $. 
The permeability $\tilde k$, the scale separation parameter $\varepsilon$ and the boundary layer constants $N_{\vec \tau}^{bl}$ and $M_{\vec \tau}^{1,bl}$ can be computed numerically based on the pore-scale information of the coupled system.
The constants $\kappa_\PM$ and $\kappa_\FF$  appearing in condition~\eqref{eq:coercivity-condition-result} are related to the  Poincar\'{e} constants from~\eqref{eq:Poincare} and \eqref{eq:kappa}. These constants depend on the size and geometry of the coupled domain~\cite{Brezis,Nazarov_15}.  
To the best of the authors' knowledge, there exist no estimates for constants $C_{\vec \tau}$ and $C_{f}$ coming from the trace inequalities \eqref{eq:trace-ff} and \eqref{eq:inequality-tangential}. 


\section{Numerical results}
\label{sec:numerics}
In this section, we first analyse the validity of condition~\eqref{eq:coercivity-condition-result} considering different porous-medium configurations. 
Then, we present a benchmark for the numerical solution of the Stokes--Darcy problem with  generalised interface conditions~\eqref{eq:IC-mass}--\eqref{eq:IC-tangential} and provide the procedure for the computation of  the material parameters.

\subsection{Analysis of the  relationship between permeability and boundary layer constants}\label{sec:validity-range}

We analyse the validity of the  relationship~\eqref{eq:coercivity-condition-result} between the permeability and the boundary layer constants for several  geometrical configurations of the coupled free-flow and porous-medium systems. We consider porous media with different shapes (circle, square, rhombus) and different sizes of solid inclusions (Tab.~\ref{tab:geoms-only}). 
Periodic porous media (Fig.~\ref{fig:domain}, right) are constructed by the periodic repetition of the scaled unit cell $\varepsilon Y=(0,\varepsilon) \times (0,\varepsilon)$. 
Note that the material parameters (permeability $\tilde k$, boundary layer constants $N_{\vec \tau}^{bl}$, $M_{\vec \tau}^{1,bl}$) depend on the shape and size of solid obstacles, but are independent on the scale separation parameter $\varepsilon$. Therefore, to analyse the validity of condition~\eqref{eq:coercivity-condition-result} for different geometrical configurations, we consider the corresponding unit cell $Y=(0,1) \times (0,1)$ and the boundary layer stripe to compute the permeability $\tilde k$ and the boundary layer constants $N_{\vec \tau}^{bl}$ and $M_{\vec \tau}^{1,bl}$, respectively (Tab.~\ref{tab:geoms-only} and Fig.~\ref{fig:BLconstants}). 

\begin{table}[h!]
    \caption{Non-dimensional permeability $\tilde k$ and squared ratio $R^2$ for different porous-medium configurations.}
    \vspace{2mm}
    \centering
    \includegraphics[scale=1]{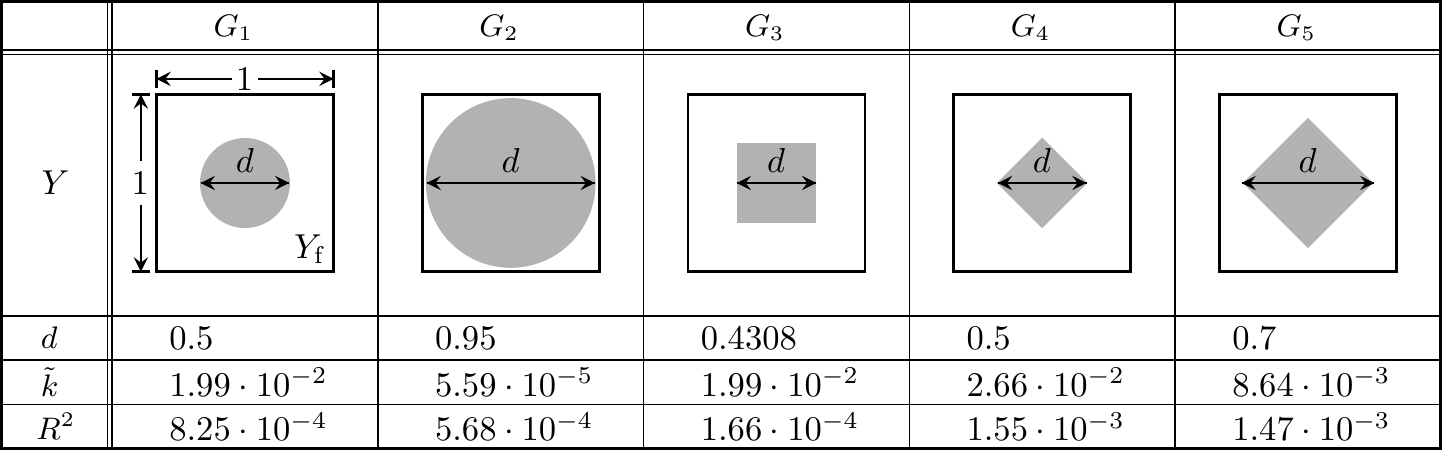}
    \label{tab:geoms-only}
\end{table}

For the sake of clarity, we reformulate condition~\eqref{eq:coercivity-condition-result} as follows
\begin{align}
    \tilde k  >  C  R^2,
    \label{eq:condition-validation}
\end{align}
where  $C:=\kappa_\PM \kappa_\FF (C_{\vec \tau} C_{f} )^2$ and $R:=M_{\vec \tau}^{1,bl}/(2 N_{\vec \tau}^{bl})$.  We can compute the permeability $\tilde k$ based on the pore geometry by means of homogenisation theory \cite{Hornung_97} using formulas~\eqref{eq:cell-problems} and \eqref{eq:permeability} from Section~\ref{sec:computation-of-BL}. 
The permeability values for the five considered geometries are presented in Table~\ref{tab:geoms-only}. 
The boundary layer constants $N_{\vec \tau}^{bl}$ and $M_{\vec \tau}^{1,bl}$ are computed by solving the boundary layer problems~\eqref{eq:BL-beta},~\eqref{eq:BL-2} and by integrating their solutions using  formulas~\eqref{eq:BL-beta-constant} and \eqref{eq:BL-t-constant}. 
The boundary layer problems~\eqref{eq:BL-beta} and \eqref{eq:BL-2} are solved numerically within a cut-off boundary layer stripe $Z^{bl}_l$, containing one column with $l=4$ solid inclusions. A schematic representation of the cut-off boundary layer stripe with $l=2$ circular solid obstacles corresponding to geometry $G_1$ is presented in Figure~\ref{fig:BLconstants} (left).
The vertical position of the interface  within this boundary layer stripe (Fig.~\ref{fig:BLconstants}, left and right) can be varied in a certain range depending on the pore geometry~\cite{Eggenweiler-MMS, Jaeger_Mikelic_00}. 
A change of the interface location within the stripe leads to the corresponding change of the boundary layer constants, which contain the information about the exact position of the interface. Note that the interface location does not influence the permeability values.
The constant $C$,  which contains the constants $C_{\vec \tau}$ and $C_{f}$ coming from the trace inequalities \eqref{eq:trace-ff} and \eqref{eq:inequality-tangential},  cannot be estimated to the best of the authors' knowledge. However, we show that there is a non-trivial validity range for the constant $C$ such that condition \eqref{eq:condition-validation} (or, equivalently, \eqref{eq:coercivity-condition-result}) is fulfilled.
To achieve this, we minimise the ratio $R$ in condition~\eqref{eq:condition-validation} by determining the optimal location of the interface.

\begin{figure}[h!]
    \centering
    \includegraphics[scale=0.8]{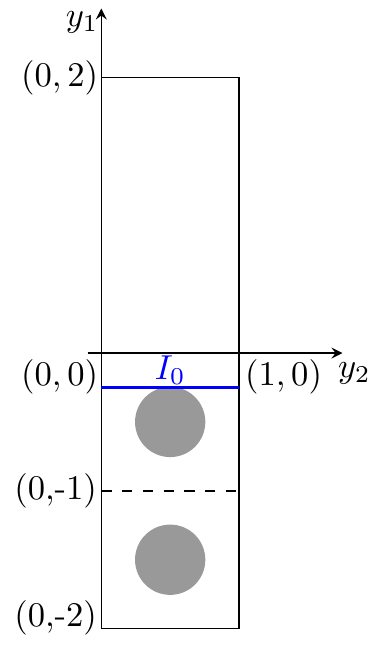}\qquad 
    \includegraphics[scale=0.92]{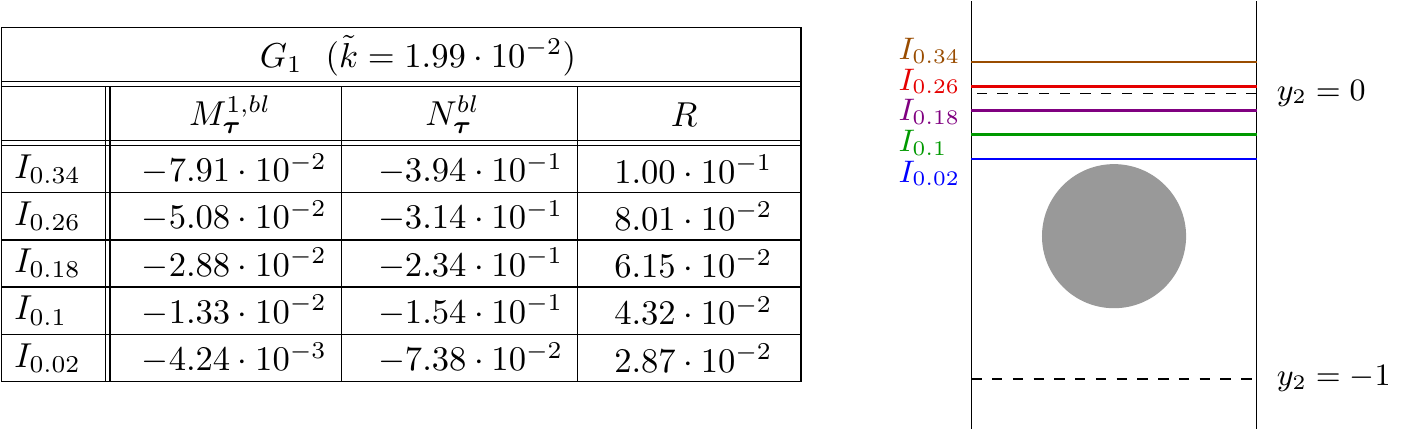}
    \caption{Boundary layer stripe $Z^{bl}_2$ for geometry $G_1$ (left), boundary layer constants and ratio $R$ for five interface locations $I_a$ (middle) and illustration of these interface locations (right).}
    \label{fig:BLconstants}
\end{figure}

We consider different vertical positions of the interface $I_a:=(0,1)\times \{a+\frac{d-1}{2}\}$, where $a\geq0$ denotes the distance between the interface $I_a$ and the top of the first row of solid inclusions (Fig.~\ref{fig:BLconstants}, right). 
In Figure~\ref{fig:BLconstants} (middle), we provide the values of the boundary layer constants $N_{\vec \tau}^{bl}$, $M_{\vec \tau}^{1,bl}$ and the ratio $R$ for five different interface positions for the geometry $G_1$. We observe that the ratio $R$ monotonically decreases as the interface location approaches the top of the solid inclusions, i.e., $a \to 0$. This correlation between $R$ and $I_a$ holds for other geometrical configurations as well. In Figure~\ref{fig:ratio} (left), we plot the ratio $R$ for the porous-medium geometries shown in Table~\ref{tab:geoms-only} versus the distance $a$, where 30 equidistantly distributed locations $I_a$  have been considered for the geometries $G_1$ and $G_3$ to $G_5$. A smaller range for $a$ has been used for geometry $G_2$ due to the bigger diameter of the solid grain. On the basis of these results, we recommend locating the interface as close as possible to the top of solid inclusions. This finding is in agreement with the experiment of Beavers and Joseph for parallel flows to the porous layer~\cite{Beavers_Joseph_67} and with the conclusion from~\cite{Rybak_etal_19}, where different flow scenarios were studied.

We compute the boundary layer constants $N_{\vec \tau}^{bl}$ and $M_{\vec \tau}^{1,bl}$ using \textsc{FreeFem++}~\cite{Hecht_12}. 
Restrictions on the meshing in \textsc{FreeFEM++} do not allow us to locate the interface directly on the top of the first row of solid inclusions. Therefore, we locate it at the next possible level, which is at distance $a=0.02$. We compute the boundary layer constants for the porous-medium configurations $G_1$ to $G_5$ presented in Table~\ref{tab:geoms-only},  considering the interface location $I_{0.02}$. The computed values are reported in Figure~\ref{fig:ratio} (right), where we also provide the ratio $R$ appearing in condition~\eqref{eq:condition-validation}. 

\begin{figure}[ht]
    \centering
    \includegraphics[scale=0.85]{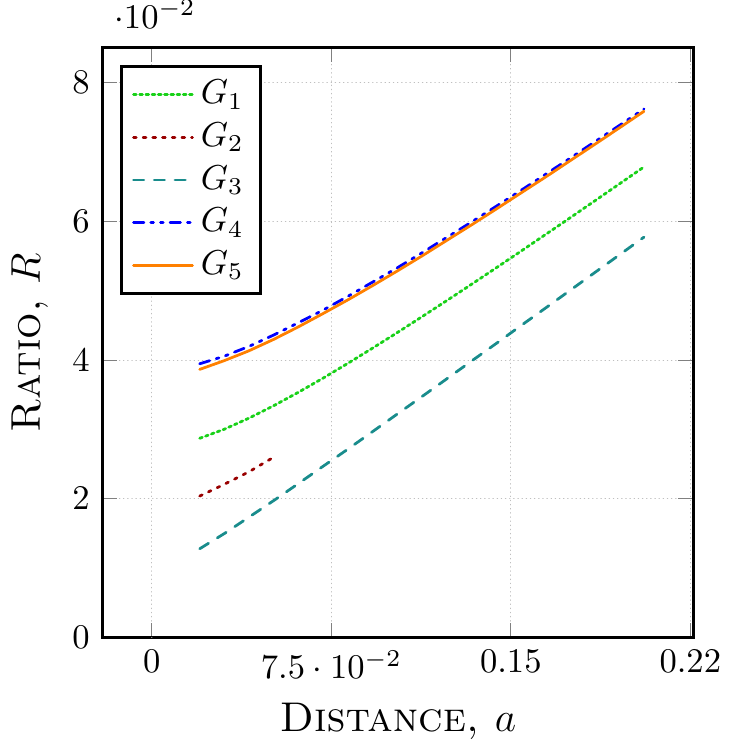} \qquad \qquad
    \includegraphics[scale=1.]{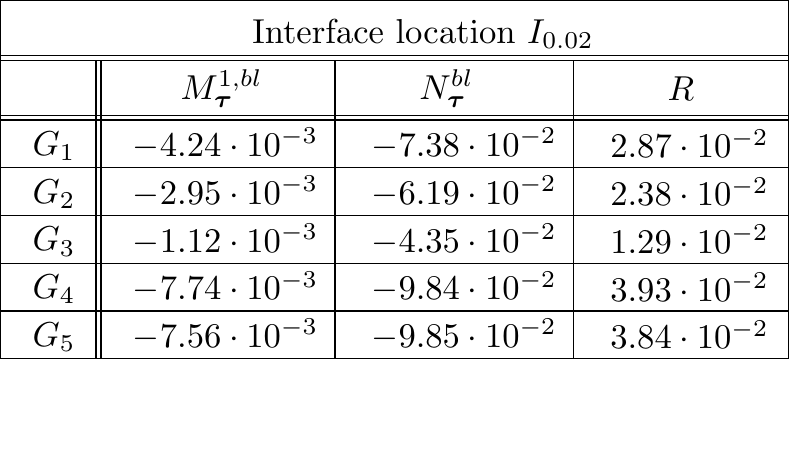}
    \caption{Ratio $R$ for geometries $G_1$ to $G_5$ computed for different interface locations (left). Boundary layer constants $N_{\vec \tau}^{bl}$, $M_{\vec \tau}^{1,bl}$  and ratio $R$ computed for the optimal location $I_{0.02}$ (right).}
    \label{fig:ratio}
\end{figure}

For the ease of analysis of condition~\eqref{eq:condition-validation}, we provide in Table~\ref{tab:geoms-only} the squared ratio $R^2$ for the optimal interface location $I_{0.02}$ next to the non-dimensional permeability $\tilde k$. For the porous-medium geometries $G_1$ and $G_3$, we have $\tilde k \gg R^2$. In this case, the constant $C$ in condition   \eqref{eq:condition-validation}  can be of order $10^2$, that is not restrictive. For the geometrical configurations $G_4$ and $G_5$, we get $\tilde k > R^2$, where the constant $C$ can be at most of order $10$. This a very mild restriction. However, for the geometry $G_2$ we obtain $\tilde k < R^2$, that requires $C < 1$, making condition~\eqref{eq:coercivity-condition-result} a stronger constraint.

\subsection{Numerical benchmark }\label{sec:numerical-benchmark}

This section serves as a benchmark for other researchers, who would like to use the generalised interface conditions~\eqref{eq:IC-mass}--\eqref{eq:IC-tangential} for numerical simulations. First, we provide the details on the computation of material parameters (permeability $\tilde k $, boundary layer constants $N_{\vec \tau}^{bl}$ and  $M_{\vec \tau}^{1,bl}$). Then, we present an analytical solution to the Stokes--Darcy problem, which satisfies the generalised interface conditions~\eqref{eq:IC-mass}--\eqref{eq:IC-tangential} at the fluid--porous interface.

\subsubsection{Computation of material parameters}\label{sec:computation-of-BL}

To compute the permeability $\tilde k$, we apply the theory of homogenisation~\cite{Hornung_97}. Since we consider isotropic porous media, it is sufficient to solve only one Stokes problem in the fluid part $Y_\text{f}$ of the unit cell $Y=(0,1)\times (0,1)$:
\begin{equation}\label{eq:cell-problems}
\begin{split}
- \Delta_ {\vec{y}} \vec{w}  &+ \nabla_{\vec{y}} \pi = 
\vec{e}_1, \quad
\operatorname{div}_{\vec{y}} \vec{w}  =0 \quad \text{in 
$Y_\text{f}$}, \quad \int_{Y_\text{f}} \pi \ \text{d} \vec{y} =0,
\\
\vec{w} =  \vec{0} & \quad  \text{on $\partial Y_\text{f} \setminus \partial Y$},  \quad
\{ \vec{w}, \pi\} \text{ is 1-periodic} \text{ in } \vec{y}, \quad \vec y = \frac{\vec x}{\varepsilon} \, ,
\end{split}
\end{equation}
where $\vec{w} = (w_1,w_2)$ and $\pi$ are the solutions to the Stokes problem  \eqref{eq:cell-problems} and $\vec e_1 =(1,0)$.
The non-dimensional permeability $\tilde k$ is then given by
\begin{equation}\label{eq:permeability}
\tilde k = \int_{Y_\text{f}} 
w_{1}(\vec y) \ \text{d}  \vec{y} \, .
\end{equation}
We solve the cell problem~\eqref{eq:cell-problems} using \textsc{FreeFem++} with Taylor--Hood finite elements. E.g., for the geometry $G_1$ presented in Table~\ref{tab:geoms-only} an adaptive mesh with approx. 50~000 elements is used.

To compute the boundary layer constants $N_{\vec \tau}^{bl}$ and $M_{\vec \tau}^{1,bl}$, the homogenisation theory with boundary layers~\cite{Carraro_etal_15,Eggenweiler-MMS,Jaeger_Mikelic_00} is applied. We consider the cut-off boundary layer stripe $Z^{bl}_4$ (see Sect.~\ref{sec:validity-range}) and define $Z^+_4= (0,1)\times(a+\frac{d-1}{2},4)$ and $Z^-_4=Z^{bl}_4 \setminus Z^+_4$. 
To obtain the boundary layer constant $M_{\vec \tau}^{1,bl} $, the following boundary layer problem is solved
\begin{equation}\label{eq:BL-beta}
\begin{split}
-\Delta_{\vec y} \vec \beta^{1,bl} + \nabla_{\vec y} \omega^{1,bl} = \vec0,  \quad 
\operatorname{div}_{\vec y} \vec \beta^{1,bl} &= 0 \quad \text{in } Z^{+}_4  \cup Z^{-}_4,
\\[1ex]
\big\llbracket \vec \beta^{1,bl} \big\rrbracket_{I_a} = - \vec w, \quad 
\big\llbracket   ( \nabla_{\vec y} \vec \beta^{1,bl} - \omega^{1,bl} \ten{I} ) \vec{e}_2\big\rrbracket_{I_a} &= - \left( \nabla_{\vec y} \vec w - \pi \ten{I} \right) \vec{e}_2 \quad \text{on } I_a,
\\[1ex]
\vec \beta^{1,bl} = \vec 0 \quad \text{on } \{ y_2 =-4 \}, 
\quad
\beta^{1,bl}_2 =   \frac{\partial\beta^{1,bl}_1}{\partial y_2} =   0 \quad \text{on } \{ y_2 =4 \}, &
\quad
\vec \beta^{1,bl} = \vec 0 \quad \text{on }  \cup_{k=1}^{4}(\partial Y_s - (0,k)) \,  ,
\end{split}
\end{equation}
where $\vec \beta^{1,bl}$ and  $\omega^{1,bl}$ are $y_1$-periodic functions and $\llbracket f\rrbracket_{I_a} := f(\cdot, +(a+\tfrac{d-1}{2}))-f(\cdot,-(a+\tfrac{d-1}{2}))$. To define $\omega^{1, bl}$ uniquely, we set $\displaystyle \int_0^1 \omega^{1, bl} (y_1, -4) \ \text{d} y_1=0$.
The boundary layer constant $M_{\vec \tau}^{1,bl}$ is then given by 
\begin{align}\label{eq:BL-beta-constant}
    M_{\vec \tau}^{1,bl} = \int_0^1 \beta_1^{1, bl}\left(y_1,a+\tfrac{d-1}{2}\right)\ \text{d} y_1 \, .
\end{align}
To compute the boundary layer constant $N_{\vec \tau}^{bl}$, we solve the following boundary layer problem
\begin{equation}\label{eq:BL-2}
\begin{split}
-\Delta_{\vec y} \vec t^{bl} + \nabla_{\vec y} s^{bl} = \vec0, \quad
\
\operatorname{div}_{\vec y} \vec t^{bl} &= 0 \quad \text{in } Z^{+}_4  \cup Z^{-}_4,  
\\[1ex]
\big\llbracket \vec t^{bl} \big\rrbracket_{I_a} = \vec 0, \quad 
\
\big\llbracket ( \nabla_{\vec y} \vec t^{bl} - s^{bl} \ten{I} ) & \vec{e}_2\big\rrbracket_{I_a} =  \vec{e}_1 \hspace{1.5ex} \text{on } I_a, 
\\[1ex]
\vec t^{bl} = \vec 0 \quad \text{on } \{ y_2 =-4 \}, 
\quad
t^{bl}_2 =   \frac{\partial t^{bl}_1}{\partial y_2} = 0 \quad \text{on } \{ y_2 =4 \},
\quad &
\vec t^{bl} = \vec 0 \quad \text{on } \cup_{k=1}^{4}(\partial Y_s - (0,k)) \, .
\end{split}
\end{equation}
Here, $\vec t^{bl}$ and  $s^{bl}$ are $y_1$-periodic. Again, for the uniqueness, we impose $\displaystyle \int_0^1  s^{bl}(y_1,-4) \ \text{d} y_1=0 $ and compute
\begin{align}
    N_{\vec \tau}^{bl} = \int_0^1  t_1^{bl} \left(y_1,a+\tfrac{d-1}{2}\right) \ \text{d} y_1 \, .
    \label{eq:BL-t-constant}
\end{align}
We solve the boundary layer problems~\eqref{eq:BL-beta} and~\eqref{eq:BL-2} numerically using \textsc{FreeFem++} with Taylor--Hood finite elements. E.g., for the boundary layer stripe with four circular inclusions $G_1$ an adaptive mesh with approx.~300~000 elements is used and the boundary layer constants are presented in~Figure~\ref{fig:BLconstants} (middle).

\subsubsection{Benchmark solution of the Stokes--Darcy problem}\label{sec:numerical-solution}

In this section, we provide an analytical solution for the Stokes--Darcy problem \eqref{eq:1p1c-FF-mass}--\eqref{eq:PM-BC} that  satisfies the generalised interface conditions~\eqref{eq:IC-mass}--\eqref{eq:IC-tangential}. This solution can serve as a benchmark for the researchers who will develop and investigate efficient numerical methods for such coupled problem. We consider the coupled domain $\Omega=(0,1)\times (0,1)$ with the interface $\Gamma =(0,1)\times \{0.5\}$ and choose the following exact solution 
\begin{equation}
\begin{split}
u_\FF  &= \sin\left(\frac{\pi x_1}2\right) \cos\left(\frac{\pi x_2}2\right)\, , \hspace{8.5ex}
p_\FF = \frac{\sqrt{2}}{2}\cos\left(\frac{\pi x_1}2\right) \left(\frac{e^{x_2-0.5}}{k}  - \frac{\pi}{2}\right) \, ,
\\
v_\FF  &=  -\cos\left(\frac{\pi x_1}2\right) \sin\left(\frac{\pi x_2}2\right) \, ,
 \hspace{5ex}
p_\PM = \frac{\sqrt{2}}{2} \cos\left(\frac{\pi x_1}2\right)\frac{e^{x_2-0.5}}{k} \, ,
\end{split}
\label{eq:exact-solution}
\end{equation}
where the free-flow velocity is $\vec v_\FF = (u_\FF, v_\FF)$.

\begin{table}[b]
    \caption{Relative errors for different grid sizes $h$. }
    \begin{center}
    \begin{tabular}{ |c||c|c|c|c| } 
        \hline 
        $h$ & \hspace{5.5ex}  $\epsilon_{u_\FF}$ \hspace{5.5ex}  & \hspace{5.5ex} $\epsilon_{v_\FF}$ \hspace{5.5ex} & \hspace{5.5ex} $\epsilon_{p_\FF}$ \hspace{5.5ex} & \hspace{5.5ex} $\epsilon_{p_\PM}$ \hspace{5.5ex} \\
        \hhline{|=||=|=|=|=|}
        $1/8$     & $5.11\mathrm{e}{+00}$    & $1.35\mathrm{e}{+00}$ &  $2.91\mathrm{e}{-03}$ & $2.29\mathrm{e}{-03}$ \\ \hline
        $1/16$    & $1.13\mathrm{e}{+00}$    & $2.81\mathrm{e}{-01}$ &   $7.66\mathrm{e}{-04}$ & $5.98\mathrm{e}{-04}$ \\ \hline
        $1/32$    & $2.73\mathrm{e}{-01}$    & $6.68\mathrm{e}{-02}$ &   $1.98\mathrm{e}{-04}$ & $1.54\mathrm{e}{-04}$ \\ \hline
        $1/64$    & $6.76\mathrm{e}{-02}$    & $1.64\mathrm{e}{-02}$ &   $5.09\mathrm{e}{-05}$ & $3.91\mathrm{e}{-05}$ \\ \hline
        $1/128$   & $1.68\mathrm{e}{-02}$    & $4.09\mathrm{e}{-03}$ &   $1.29\mathrm{e}{-05}$ & $1.00\mathrm{e}{-05}$ \\ \hline
        $1/256$   & $4.21\mathrm{e}{-03}$    & $1.02\mathrm{e}{-03}$ &   $3.28\mathrm{e}{-06}$ & $2.52\mathrm{e}{-06}$ \\ \hline
        $1/512$   & $1.05\mathrm{e}{-03}$    & $2.54\mathrm{e}{-04}$ &  $8.22\mathrm{e}{-07}$ & $6.33\mathrm{e}{-07}$ \\ \hline
        $1/1024$  & $2.63\mathrm{e}{-04}$    & $6.39\mathrm{e}{-05}$ &   $2.05\mathrm{e}{-07}$ & $1.59\mathrm{e}{-07}$ \\
        \hline
    \end{tabular} \label{tab:convergence}
    \end{center}
\end{table}

We choose the realistic value of permeability $k=10^{-6}$ and consider $\varepsilon = 10^{-1}$. We set the force terms in the right-hand sides of the Stokes and Darcy equations as well as the Dirichlet boundary conditions by substitution of the exact solution \eqref{eq:exact-solution} and the permeability value into the model formulation~\eqref{eq:1p1c-FF-mass}--\eqref{eq:PM-BC}. The exact solution~\eqref{eq:exact-solution} satisfies the generalised  interface conditions~\eqref{eq:IC-mass}--\eqref{eq:IC-tangential} taking the boundary layer constants $N_{\vec \tau}^{bl}= -\tfrac{1}{\pi}\approx - 0.3183$ and $M_{\vec \tau}^{1,bl}= -\tfrac{2k(1+0.5\varepsilon)}{\pi\varepsilon^2}\approx - 0.06398$. We note that these values lie in a typical range of the boundary layer constants for realistic geometries (Fig.~\ref{fig:BLconstants}, middle and Fig.~\ref{fig:ratio}, right).
Remember that $M_{\vec \tau}^{2,bl}=0$ and $N_s^{bl}=0$ due to isotropy of the porous medium.

We solve the coupled problem~\eqref{eq:1p1c-FF-mass}--\eqref{eq:IC-tangential} numerically using the second-order finite volume method on staggered grids. In both flow domains, we consider uniform rectangular meshes conforming at the fluid--porous interface~$\Gamma$. Since the chosen discretisation scheme is of second order, we decrease the grid step by the factor of two at each level of refinement starting with $h=1/8$. We perform seven levels of grid refinement and compute the relative $\xLn{2}$-errors  for all primary variables
\begin{align*}
\epsilon_f = \frac{\| f-f_h\|_{0,i}}{\|f\|_{0,i}}, \qquad f\in\{u_\FF, \ v_\FF, \ p_\FF, \ p_\PM\},
\end{align*}
where $f_h$ is the numerical solution and $i\in\{\FF,\PM\}$ depending on the primary variable $f$. The numerical simulation results are presented in Table~\ref{tab:convergence} and Figure~\ref{fig:convergence-model}, that demonstrates the second order convergence of the discretisation scheme. 

\enlargethispage{0.75cm}

\begin{figure}[h!] 
    \centering
    \includegraphics[scale=0.8]{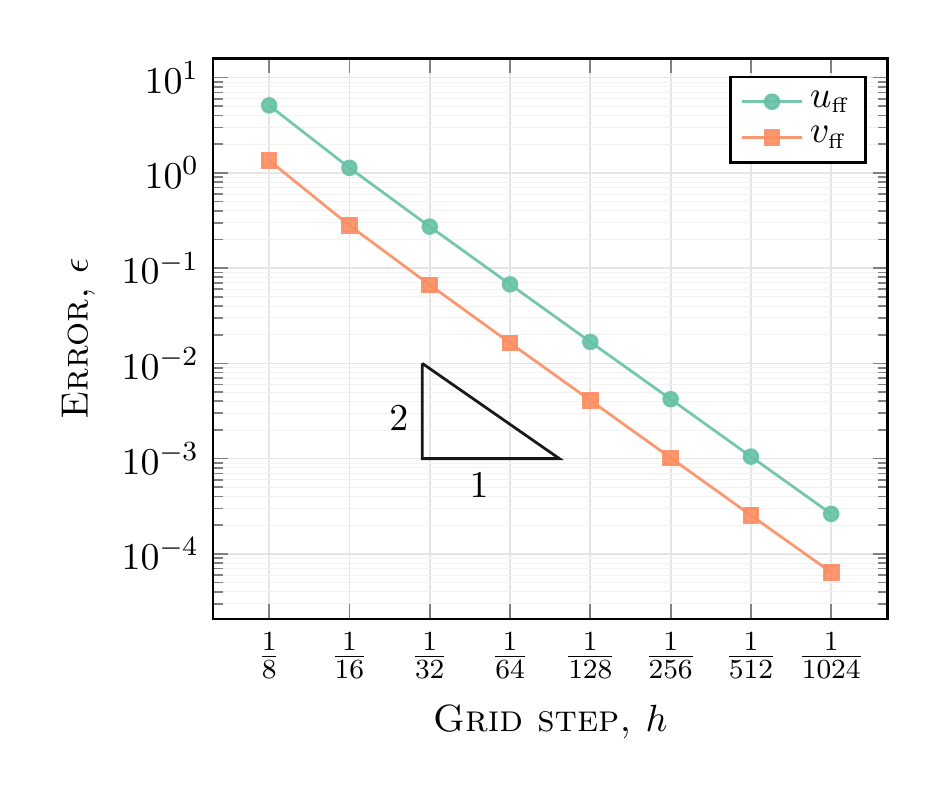} \qquad
    \includegraphics[scale=0.8]{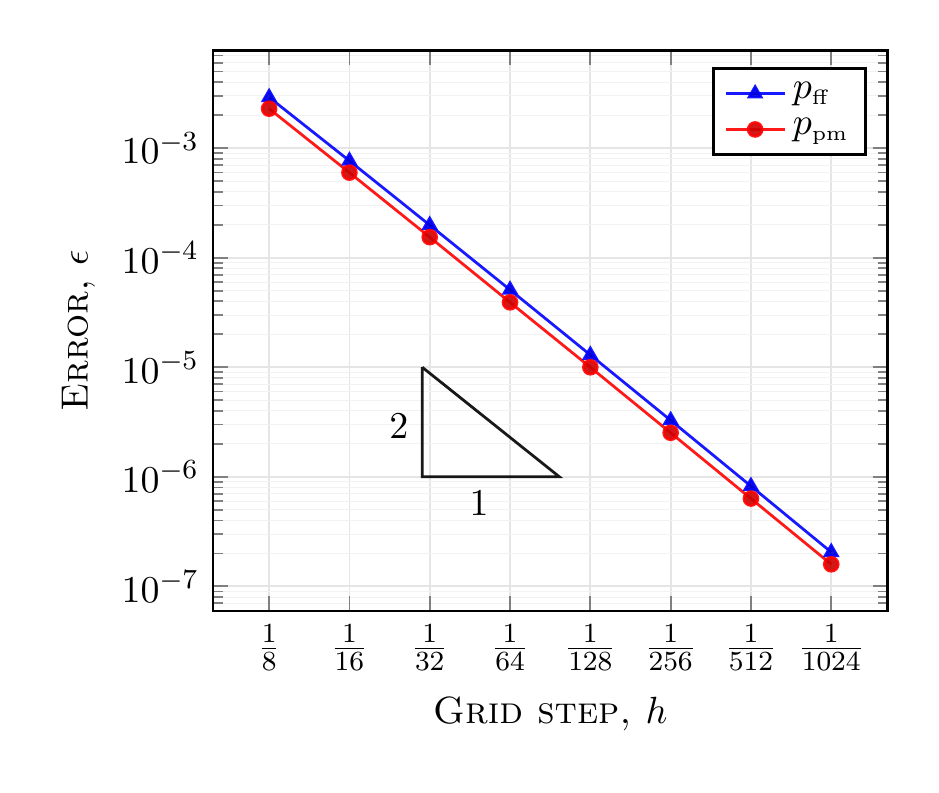}
    \caption{Error analysis for all primary variables.}
    \label{fig:convergence-model}
\end{figure}

\section{Summary}\label{sec:summary}

In this paper, we have analysed the Stokes--Darcy problem with generalised coupling conditions at the interface between the fluid and the porous-medium domains. These conditions extend the classical coupling conditions based on the Beavers--Joseph condition to the case of arbitrary flows non-parallel to the interface. We have proved that the resulting coupled problem is well-posed under the hypothesis that the permeability is large enough compared to the boundary layer constants that take into account the geometrical properties of the porous medium in the interfacial region.

Numerical tests have been used to identify the optimal position of the interface in order to guarantee that the constraint on the permeability becomes the least restrictive possible for a range of porous-medium geometries. Finally, we have given practical indications on how to compute the boundary layer constants, and we have provided a benchmark test case for the Stokes--Darcy problem with the generalised coupling conditions.

Future extensions of this work will focus on the development and analysis of effective decoupling algorithms to solve the Stokes--Darcy problem with the generalised interface conditions.

\section*{Acknowledgments}
The work is funded by the Deutsche Forschungsgemeinschaft (DFG, German Research Foundation) -- Project Number 327154368 -- SFB 1313.\\
The authors thank Profs. Ana Alonso and Alberto Valli for the helpful discussions on trace inequalities in $\mathrm{H}(\mbox{curl};\Omega)$ spaces.

\bibliographystyle{abbrv}
\bibliography{EDR-arxiv.bib}

\begin{thebibliography}{10}

\bibitem{Angot_18}
P.~Angot.
\newblock Well-posed {S}tokes/{B}rinkman and {S}tokes/{D}arcy coupling
  revisited with new jump interface conditions.
\newblock {\em ESAIM. Math. Model. Numer. Anal.}, 52:1875--1911, 2018.

\bibitem{Angot_etal_17}
P.~Angot, B.~Goyeau, and J.~A. Ochoa-Tapia.
\newblock Asymptotic modeling of transport phenomena at the interface between a
  fluid and a porous layer: jump conditions.
\newblock {\em Phys. Rev. E}, 95:063302, 2017.

\bibitem{Beaude_etal_19}
L.~Beaude, K.~Brenner, S.~Lopez, R.~Masson, and F.~Smai.
\newblock Non-isothermal compositional liquid gas {D}arcy flow: formulation,
  soil-atmosphere boundary condition and application to high-energy geothermal
  simulations.
\newblock {\em Comput. Geosci.}, 23:443--470, 2019.

\bibitem{Beavers_Joseph_67}
G.~S. Beavers and D.~D. Joseph.
\newblock Boundary conditions at a naturally permeable wall.
\newblock {\em J. Fluid Mech.}, 30:197--207, 1967.

\bibitem{Brezis}
H.~Brezis.
\newblock {\em Functional Analysis, Sobolev Spaces and Partial Differential
  Equations}.
\newblock Springer, 2011.

\bibitem{Brezzi:1974:OEU}
F.~Brezzi.
\newblock On the existence, uniqueness and approximation of saddle-point
  problems arising from {L}agrange multipliers.
\newblock {\em Rev. Fran\c caise Automat. Informat. Recherche Op\'erationnelle,
  s\'er. Rouge}, 8:129--151, 1974.

\bibitem{Cao_10}
Y.~Cao, M.~Gunzburger, F.~Hua, and X.~Wang.
\newblock Coupled {S}tokes--{D}arcy model with {B}eavers--{J}oseph interface
  boundary condition.
\newblock {\em Commun. Math. Sci.}, 8:1--25, 2010.

\bibitem{Carraro_etal_15}
T.~Carraro, C.~Goll, A.~Marciniak-Czochra, and A.~Mikeli\'{c}.
\newblock Effective interface conditions for the forced infiltration of a
  viscous fluid into a porous medium using homogenization.
\newblock {\em Comput. Methods Appl. Mech. Engrg.}, 292:195--220, 2015.

\bibitem{Dautray_90}
R.~Dautray and J.~Lions.
\newblock {\em Mathematical Analysis and Numerical Methods for Science and
  Technology}.
\newblock Springer, 1990.

\bibitem{Dawson_08}
C.~Dawson.
\newblock A continuous/discontinuous {G}alerkin framework for modeling coupled
  subsurface and surface water flow.
\newblock {\em Comput. Geosci.}, 12:451--472, 2008.

\bibitem{Discacciati_Miglio_Quarteroni_02}
M.~Discacciati, E.~Miglio, and A.~Quarteroni.
\newblock Mathematical and numerical models for coupling surface and
  groundwater flows.
\newblock {\em Appl. Num. Math.}, 43:57--74, 2002.

\bibitem{Discacciati_Quarteroni_09}
M.~Discacciati and A.~Quarteroni.
\newblock Navier--{S}tokes/{D}arcy coupling: modeling, analysis, and numerical
  approximation.
\newblock {\em Rev. Mat. Complut.}, 22:315--426, 2009.

\bibitem{Eggenweiler_Rybak_20}
E.~Eggenweiler and I.~Rybak.
\newblock Unsuitability of the {B}eavers--{J}oseph interface condition for
  filtration problems.
\newblock {\em J. Fluid Mech.}, 892:A10, 2020.

\bibitem{Eggenweiler-MMS}
E.~Eggenweiler and I.~Rybak.
\newblock Effective coupling conditions for arbitrary flows in {S}tokes-{D}arcy
  systems.
\newblock {\em Multiscale Model. Simul. (in press)}, 2021.
\newblock https://arxiv.org/abs/2006.12096v1.

\bibitem{GiraultRiviere}
V.~Girault and B.~Rivi\`{e}re.
\newblock {DG} approximation of coupled {N}avier--{S}tokes and {D}arcy
  equations by {B}eavers--{J}oseph--{S}affman interface condition.
\newblock {\em SIAM J. Numer. Anal.}, 47:2052--2089, 2009.

\bibitem{Goyeau_Lhuillier_etal_03}
B.~Goyeau, D.~Lhuillier, D.~Gobin, and M.~Velarde.
\newblock Momentum transport at a fluid-porous interface.
\newblock {\em Int. J. Heat Mass Transfer}, 46:4071--4081, 2003.

\bibitem{Hanspal_etal_09}
N.~Hanspal, A.~Waghode, V.~Nassehi, and R.~Wakeman.
\newblock Development of a predictive mathematical model for coupled
  {S}tokes/{D}arcy flows in cross-flow membrane filtration.
\newblock {\em Chem. Eng. J.}, 149:132--142, 2009.

\bibitem{Hecht_12}
F.~Hecht.
\newblock New development in \textsc{{F}ree{F}em++}.
\newblock {\em J. Numer. Math.}, 20:251--265, 2012.

\bibitem{Hornung_97}
U.~Hornung.
\newblock {\em Homogenization and Porous Media}.
\newblock Springer, 1997.

\bibitem{Hou_Qin_19}
Y.~Hou and Y.~Qin.
\newblock On the solution of coupled {S}tokes/{D}arcy model with
  {B}eavers--{J}oseph interface condition.
\newblock {\em Comput. Math. Appl.}, 77:50--65, 2019.

\bibitem{Jaeger_Mikelic_00}
W.~J\"{a}ger and A.~Mikeli\'{c}.
\newblock On the interface boundary conditions by {B}eavers, {J}oseph and
  {S}affman.
\newblock {\em SIAM J. Appl. Math.}, 60:1111--1127, 2000.

\bibitem{Jaeger_Mikelic_09}
W.~J\"{a}ger and A.~Mikeli\'{c}.
\newblock Modeling effective interface laws for transport phenomena between an
  unconfined fluid and a porous medium using homogenization.
\newblock {\em Transp. Porous Media}, 78:489--508, 2009.

\bibitem{Jarauta_etal_20}
A.~Jarauta, V.~Zingan, P.~Minev, and M.~Secanell.
\newblock A compressible fluid flow model coupling channel and porous media
  flows and its application to fuel cell materials.
\newblock {\em Transp. Porous Med.}, 134:351--386, 2020.

\bibitem{Kanschat_Riviere_10}
G.~Kanschat and B.~Rivi\`{e}re.
\newblock A strongly conservative finite element method for the coupling of
  {S}tokes and {D}arcy flow.
\newblock {\em J. Comput. Phys.}, 229:5933--5943, 2010.

\bibitem{Lacis_Bagheri_17}
U.~L\={a}cis and S.~Bagheri.
\newblock A framework for computing effective boundary conditions at the
  interface between free fluid and a porous medium.
\newblock {\em J. Fluid Mech.}, 812:866--889, 2017.

\bibitem{Lacis_etal_20}
U.~L\={a}cis, Y.~Sudhakar, S.~Pasche, and S.~Bagheri.
\newblock Transfer of mass and momentum at rough and porous surfaces.
\newblock {\em J. Fluid Mech.}, 884:A21, 2020.

\bibitem{Layton_Schieweck_Yotov_03}
W.~Layton, F.~Schieweck, and I.~Yotov.
\newblock Coupling fluid flow with porous media flow.
\newblock {\em SIAM J. Numer. Anal.}, 40:2195--2218, 2003.

\bibitem{Bars_Worster_06}
M.~Le~Bars and M.~Worster.
\newblock Interfacial conditions between a pure fluid and a porous medium:
  implications for binary alloy solidification.
\newblock {\em J. Fluid Mech.}, 550:149--173, 2006.

\bibitem{Lions_68}
J.~Lions and E.~Magenes.
\newblock {\em Non-Homogeneous Boundary Problemes and Applications}.
\newblock Springer, 1972.

\bibitem{Maxwell2014}
R.~M. Maxwell, M.~Putti, S.~Meyerhoff, J.-O. Delfs, I.~M. Ferguson, V.~Ivanov,
  J.~Kim, O.~Kolditz, S.~J. Kollet, M.~Kumar, S.~Lopez, J.~Niu, C.~Paniconi,
  Y.-J. Park, M.~S. Phanikumar, C.~Shen, E.~A. Sudicky, and M.~Sulis.
\newblock Surface-subsurface model intercomparison: {A} first set of benchmark
  results to diagnose integrated hydrology and feedbacks.
\newblock {\em Water Resour. Res.}, 50:1531--1549, 2014.

\bibitem{Mosthaf_Baber_etal_11}
K.~Mosthaf, K.~Baber, B.~Flemisch, R.~Helmig, A.~Leijnse, I.~Rybak, and
  B.~Wohlmuth.
\newblock A coupling concept for two-phase compositional porous-medium and
  single-phase compositional free flow.
\newblock {\em Water Resour. Res.}, 47:W10522, 2011.

\bibitem{Nazarov_15}
A.~I. Nazarov and S.~I. Repin.
\newblock Exact constants in {P}oincar\'{e} type inequalities for functions
  with zero mean boundary traces.
\newblock {\em Math. Meth. Appl. Sci.}, 38:3195--3207, 2015.

\bibitem{Nield_09}
D.~A. Nield.
\newblock The {B}eavers--{J}oseph boundary condition and related matters: a
  historical and critical note.
\newblock {\em Transp. Porous Media}, 78:537--540, 2009.

\bibitem{OchoaTapia_Whitaker_95}
A.~J. Ochoa-Tapia and S.~Whitaker.
\newblock Momentum transfer at the boundary between a porous medium and a
  homogeneous fluid. {I}: Theoretical development.
\newblock {\em Int. J. Heat Mass Transfer}, 38:2635--2646, 1995.

\bibitem{Reuter_etal_19}
B.~Reuter, A.~Rupp, V.~Aizinger, and P.~Knabner.
\newblock Discontinuous {G}alerkin method for coupling hydrostatic free surface
  flows to saturated subsurface systems.
\newblock {\em Comput. Math. Appl.}, 77:2291--2309, 2019.

\bibitem{Rybak_etal_15}
I.~Rybak, J.~Magiera, R.~Helmig, and C.~Rohde.
\newblock Multirate time integration for coupled saturated/unsaturated porous
  medium and free flow systems.
\newblock {\em Comput. Geosci.}, 19:299--309, 2015.

\bibitem{Rybak_etal_19}
I.~Rybak, C.~Schwarzmeier, E.~Eggenweiler, and U.~R\"{u}de.
\newblock Validation and calibration of coupled porous-medium and free-flow
  problems using pore-scale resolved models.
\newblock {\em Comput. Geosci.}, 2020.
\newblock doi: 10.1007/s10596-020-09994-x.

\bibitem{Saffman_71}
P.~G. Saffman.
\newblock On the boundary condition at the surface of a porous medium.
\newblock {\em Stud. Appl. Math.}, 50:93--101, 1971.

\bibitem{Sochala_Ern_Piperno_09}
P.~Sochala, A.~Ern, and S.~Piperno.
\newblock Mass conservative {BDF}-discontinuous {G}alerkin/explicit finite
  volume schemes for coupling subsurface and overland flows.
\newblock {\em Comput. Methods Appl. Mech. Engrg.}, 198:2122--2136, 2009.

\bibitem{Tartar}
L.~Tartar.
\newblock {\em An Introduction to Sobolev Spaces and Interpolation Spaces}.
\newblock Lecture Notes of the Unione Matematica Italiana, 2007.

\bibitem{Zampogna_Bottaro_16}
G.~A. Zampogna and A.~Bottaro.
\newblock Fluid flow over and through a regular bundle of rigid fibres.
\newblock {\em J. Fluid Mech.}, 792:5--35, 2016.

\end{thebibliography}
\end{document}